\documentclass{amsart}
\usepackage{amssymb}
\usepackage{mathrsfs}

\newtheorem{theorem}{Theorem}[section]
\newtheorem{lemma}[theorem]{Lemma}
\newtheorem{e-proposition}[theorem]{Proposition}
\newtheorem{corollary}[theorem]{Corollary}
\newtheorem{e-definition}[theorem]{Definition}

\newtheorem{example}[theorem]{\it Example\/}

\DeclareMathOperator{\Hom}{Hom}
\DeclareMathOperator{\id}{id}

\DeclareMathOperator{\pbw}{pbw}

\DeclareMathOperator{\End}{End}

\DeclareMathOperator{\Todd}{Td}
\DeclareMathOperator{\str}{str}
\DeclareMathOperator{\CE}{CE}
\DeclareMathOperator{\Ber}{Ber}

\DeclareMathOperator{\At}{At}

\newcommand{\PBW}{{\pbw^\nabla}}

\newcommand{\abs}[1]{\left| #1 \right|} % absolute value

\newcommand{\isomorphism}{\cong}
\newcommand{\xto}[1]{\xrightarrow{#1}}
\newcommand{\NN}{\mathbb{N}}
\newcommand{\ZZ}{\mathbb{Z}}
\newcommand{\RR}{\mathbb{R}}
\newcommand{\CC}{\mathbb{C}}

\newcommand{\sections}[1]{\Gamma(#1)}

\newcommand{\lie}[2]{[#1,#2]} % Lie bracket

\newcommand{\frakg}{{\mathfrak g}}

\newcommand{\hatS}{\hat{S}}

\newcommand{\degree}[1]{\abs{#1}}

\newcommand{\cM}{\mathcal{M}}
\newcommand{\cA}{\mathcal{A}}
\newcommand{\cald}{\mathcal{D}}
\newcommand{\cE}{\mathcal{E}}
\newcommand{\cR}{\mathcal{R}}
\newcommand{\vect}{\mathfrak{X}}
\newcommand{\XX}{\vect}
\newcommand{\op}{\mathrm{op}}

\newcommand{\Aat}{\alpha}

\newcommand{\T}{\mathcal{T}}

\begin{document}

\title{The Atiyah class of a dg-vector bundle}

\author{Rajan Amit Mehta}
\address{Department of Mathematics \& Statistics\\Smith College\\44 College Lane\\Northampton, MA 01063}
\email{rmehta@smith.edu}
\author{Mathieu Sti\'enon}
\address{Department of Mathematics\\Pennsylvania State University\\University Park, PA 16802}
\email{stienon@psu.edu}
\author{Ping Xu}
\email{ping@math.psu.edu}
\thanks{{Research partially supported by NSF grants DMS1406668, and NSA grants H98230-06-1-0047 
and H98230-14-1-0153.}}

\begin{abstract}
We introduce the notions of
Atiyah class and Todd class of a differential graded vector bundle with respect to a
differential graded Lie algebroid. We prove that the space of vector fields $\XX(\cM)$
on a dg-manifold $\cM$ with homological vector field $Q$ admits a structure of $L_\infty[1]$-algebra 
with the Lie derivative $L_Q$ as unary bracket $\lambda_1$, and the Atiyah cocycle 
$\At_\cM$ corresponding to a torsion-free affine connection as binary bracket $\lambda_2$.
\end{abstract}

\dedicatory{En hommage \`a Charles-Michel Marle \`a l'occasion de son quatre-vingti\`eme anniversaire}

\maketitle

\section{dg-manifolds and dg-vector bundles}

A $\ZZ$-graded manifold $\cM$ with base manifold $M$ is
a sheaf of $\ZZ$-graded, graded-commutative algebras 
$\{\cR_U|U\subset M \ \text{open}\}$
over $M$,  locally isomorphic to $C^\infty (U)\otimes
\hat{S}(V^\vee)$, where $U\subset M$ is an open submanifold,
$V$ is a $\ZZ$-graded vector space, and $\hat{S}(V^\vee)$ denotes
the graded  algebra of formal polynomials on $V$. 
By $C^\infty (\cM)$, we denote the $\ZZ$-graded, graded-commutative algebra
of global sections.  By a dg-manifold, 
we mean a $\ZZ$-graded manifold endowed with a
homological vector field, i.e.\ a vector field $Q$ of degree $+1$ satisfying
$\lie{Q}{Q}=0$.

\begin{example}
Let $A\to M$ be a Lie algebroid over $\CC$. Then $A[1]$ is a dg-manifold 
with the Chevalley--Eilenberg differential $d_{\CE}$ as homological vector field.
In fact, according to Va\u{\i}ntrob~\cite{Vaintrob},
there is a bijection between the Lie algebroid structures on the vector bundle $A\to M$ 
and the homological vector fields on the $\ZZ$-graded manifold $A[1]$.
\end{example}

\begin{example}
Let s be a smooth section of a vector bundle $E\to M$.
Then $E[-1]$ is a dg-manifold with the contraction operator $i_s$ as homological vector field.
\end{example}

\begin{example}
Let $\frakg=\sum_{i\in \ZZ}\frakg_i$ be a $\ZZ$-graded vector space of finite type, 
i.e.\ each $\frakg_i$ is a finite-dimensional vector space.
Then $\frakg[1]$ is a dg-manifold if and only if $\frakg$ is an $L_\infty$-algebra.
\end{example}

A dg-vector bundle is a vector bundle in
the category of dg-manifolds.  We refer the reader to~\cite{Mehta-Qalgebroids,Kotov-Strobl} for
details on dg-vector bundles.
The following example is essentially due to Kotov--Strobl~\cite{Kotov-Strobl}.

\begin{example}
Let $A\to M$ be a gauge Lie  algebroid with anchor $\rho$.
Then $A[1]\to T[1]M$ is a dg-vector bundle, where
the homological vector fields on  $A[1]$ and 
$T[1]M$ are the Chevalley--Eilenberg differentials.
\end{example} 

The example above is a special case of a general fact~\cite{Mehta-Qalgebroids}, that LA-vector bundles~\cite{Kirill-11,stuff,Kirill-15} (also known as VB-algebroids~\cite{G-SM:VBalg}) give rise to dg-vector bundles. 

Given a vector bundle $\cE\stackrel{\pi}{\to} \cM$ of graded
manifolds, its space of sections, denoted
$\sections{\cE}$, is  defined to be $\bigoplus_{j\in \ZZ}\Gamma_j (\cE)$,
where $\Gamma_j (\cE)$ consists of  degree
preserving maps $s\in \Hom (\cM, \cE[-j])$ such that
 $(\pi[-j])\circ s=\id_\cM$, where
$\pi[-j]: \cE[-j]\to \cM$ is the natural map induced from $\pi$; see~\cite{Mehta-Qalgebroids} for more details.
When $\cE\to \cM$ is a dg-vector bundle, the homological
vector fields on $\cE$ and $\cM$ naturally induce
a degree $1$ operator $\mathcal{Q}$ on $\sections{\cE}$, making  
$\sections{\cE}$ a dg-module over $C^\infty(\cM)$. 
Since the space $\sections{\cE^\vee}$ of linear functions on $\cE$ generates
$C^\infty(\cE)$, the converse is also true.

\begin{lemma}
Let $\cE\to\cM$ be a vector bundle object in the category of graded manifolds and suppose $\cM$ is a dg-manifold. 
If $\Gamma(\cE)$  is a dg-module over $C^\infty (\cM)$, then $\cE$ admits a natural dg-manifold structure such that $\cE\to\cM$ is a dg-vector bundle.
In fact, the categories of dg-vector bundles and of locally free dg-modules are equivalent.
\end{lemma}

In this case, the degree $+1$ operator $\mathcal{Q}$ on $\Gamma(\cE)$ gives rise to a cochain complex
\[ \cdots\to\Gamma_i (\cE)\xto{\mathcal{Q}}\Gamma_{i+1}(\cE)\to\cdots ,\]
whose cohomology group will be denoted by $H^\bullet(\Gamma(\cE),\mathcal{Q})$.

In particular, the space $\vect (\cM)$ of vector fields on a dg-manifold $(\cM,Q)$
(i.e.\ graded derivations of $C^\infty (\cM)$), which can be regarded as 
the space of sections $\Gamma(T\cM)$, is naturally a dg-module over
$C^\infty (\cM)$ with the Lie derivative $L_Q:\vect (\cM)\to\vect (\cM)$ playing the role of the degree $+1$ operator $\mathcal{Q}$. 

Thus we have the following

\begin{corollary}\label{cow}
For every dg-manifold $(\cM,Q)$, the Lie derivative $L_Q$ makes $\sections{T\cM}$ into a dg-module over $C^\infty(\cM)$ 
and therefore $T\cM\to\cM$ is naturally a dg-vector bundle. 
\end{corollary}

Following the classical case, the corresponding homological vector field
on $T\cM$ is called the {\em tangent lift} of $Q$.

Differential graded Lie algebroids are another useful notion. 
Roughly, a dg-Lie algebroid can be thought of as a Lie algebroid object in the category of dg-manifolds. For more details, we refer the reader to~\cite{Mehta-Qalgebroids}, where dg-Lie algebroids are called $Q$-algebroids.

Differential graded foliations constitute an important class of examples of dg-Lie algebroids.
\begin{lemma}\label{donkey}
Let $\cald\subset T\cM$ be an integrable distribution on a graded manifold $\cM$. 
Suppose there exists a homological vector field $Q$ on $\cM$ such that $\sections{\cald}$ is stable under $L_Q$. 
Then $\cald\to\cM$ is a dg-Lie algebroid with the inclusion $\rho:\cald\to T\cM$ as its anchor map. 
\end{lemma}

\section{Atiyah class and Todd class of a dg-vector bundle}

Let $\cE\to\cM$ be a dg-vector bundle and let $\cA\to\cM$ be a dg-Lie algebroid with anchor $\rho:\cA\to T\cM$. An $\cA$-connection on $\cE\to\cM$ is a degree $0$ map $\nabla:\sections{\cA}\otimes\sections{\cE}\to\sections{\cE}$ such that
\begin{gather*}
\nabla_{fX} s = f \nabla_X s \\ 
\intertext{and} 
\nabla_X (fs) = \rho(X)(f) s + (-1)^{|X||f|} f \nabla_X s
\end{gather*}
for all $f\in C^\infty(\cM)$, $X\in\sections{\cA}$, and $s\in\sections{\cE}$. Here we use the `absolute value' notation to denote the degree of the argument. 
When we say that $\nabla$ is of  degree $0$, we actually mean that $|\nabla_X s| = |X| + |s|$ for every pair of homogeneous elements $X$ and $s$.
Such connections always exist since the standard partition of unity argument holds in the context of graded manifolds.  Given a dg-vector bundle $\cE\to\cM$ and an $\cA$-connection $\nabla$ on it, 
we can consider the bundle map $\At_\cE:\cA\otimes\cE\to\cE$ defined by 
\begin{equation}\label{eq:At}
\At_\cE (X,s) := \mathcal{Q}(\nabla_X s) -\nabla_{\mathcal{Q}(X)}s -(-1)^{|X|}\nabla_X\big(\mathcal{Q}(s)\big), 
\quad \forall X\in\sections{\cA}, s\in\sections{\cE}.
\end{equation}

\begin{e-proposition}
\begin{enumerate}
\item $\At_\cE:\cA\otimes\cE\to\cE$ is a degree $+1$ bundle map
and therefore can also be regarded as a degree $+1$ section of 
$\cA^\vee\otimes\End\cE$. 
\item $\At_\cE$ is a cocycle: $\mathcal{Q}(\At_\cE)=0$.
\item The cohomology class of $\At_\cE$ is independent of the choice of the connection $\nabla$.
\end{enumerate}
\end{e-proposition}

Thus there is a natural cohomology class $\Aat_\cE := [\At_\cE]$ in $H^1\big(\Gamma(\cA^\vee\otimes\End\cE),Q\big)$. The class $\Aat_\cE$ is
called the \emph{Atiyah class} of the dg-vector bundle $\cE\to\cM$ relative to the dg-Lie algebroid $\cA\to\cM$.

The Atiyah class of a dg-manifold, which is the obstruction to the existence of connections compatible with the differential, 
was first investigated by Shoikhet~\cite{Shoikhet} in relation with Kontsevich's formality theorem and Duflo formula.
More recently, the Atiyah class of a dg-vector bundle appeared in Costello's work~\cite{MR2827826}.

We define the \emph{Todd class} $\Todd_\cE$ of a dg-vector bundle $\cE\to\cM$ relative to a dg-Lie algebroid $\cA\to\cM$ by 
\begin{equation}
\Todd_\cE:=\Ber\left(\frac{1-e^{-\Aat_\cE}}{\Aat_\cE}\right)\in\prod_{k\geq 0} H^k\big(\sections{\wedge^k\cA^\vee},Q\big)
,\end{equation}
where $\Ber$ denotes the Berezinian~\cite{Manin} and $\wedge^k\cA^\vee$ denotes the dg vector bundle $S^k (\cA^\vee[-1])[k]\to\cM$. One checks that $\Todd_{\cE}$ can be expressed in terms of scalar Atiyah classes
$c_k=\frac{1}{k!}(\frac{i}{2\pi})^k \str \Aat_\cE^k \in H^k\big(\sections{\wedge^k\cA^\vee},Q\big)$.
Here $\str: \End \cE\to C^\infty (\cM)$ denotes the supertrace.
Note that $\str\Aat_\cE^k\in\sections{\wedge^k\cA^\vee}$ since $\Aat_\cE^k\in\sections{\wedge^k\cA^\vee}\otimes_{C^\infty(\cM)}\End\cE$. 
If $\cA=T\cM$, we write $\Omega^k(\cM)$ instead of $\sections{\wedge^k T^\vee\cM}$.

\section{Atiyah class and Todd class of a dg-manifold}

Consider a dg-manifold $(\cM, Q)$. According to Lemma~\ref{donkey},
 its tangent bundle  $T\cM$ is indeed a dg-Lie algebroid.
By the {\it Atiyah class of a dg-manifold} $(\cM, Q)$, 
denoted  $\Aat_\cM$, we  mean the  Atiyah class  of the  tangent dg-vector
 bundle $T\cM\to \cM$ with respect to the dg-Lie algebroid $T\cM$.
Similarly, the Atiyah $1$-cocycle of a dg manifold $\cM$ corresponding to an affine connection on $\cM$ 
(i.e.\ a $T\cM$-connection on $T\cM\to\cM$) is the $1$-cocycle defined as in Eq.~\eqref{eq:At}. 

\begin{lemma}\label{dog}
Suppose $V$ is a vector space.
The only connection on the graded manifold $V[1]$ is the trivial connection.
\end{lemma}

\begin{proof}
Since the graded algebra of functions on $V[1]$ is $\wedge(V^\vee)$, 
every vector $v\in V$ determines a degree $-1$ vector field $\iota_v$ on $V[1]$, 
which acts as a contraction operator on $\wedge(V^\vee)$. The $C^\infty(V[1])$-module of all vector fields on $V[1]$ is generated by its subset $\{\iota_v\}_{v\in V}$. It follows that a connection $\nabla$ on $V[1]$ is completely determined by 
the knowledge of $\nabla_{\iota_v}\iota_w$ for all $v,w\in V$.
Since the degree of every vector field $\nabla_{\iota_v}\iota_w$ must be $-2$ and there are no nonzero vector fields of degree $-2$, it follows
that $\nabla_{\iota_v}\iota_w=0$.
\end{proof}

Given a finite-dimensional Lie algebra $\frakg$, consider the dg-manifold $(\cM,Q)$, where $\cM=\frakg[1]$ and $Q$ is the Chevalley-Eilenberg differential $d_{\CE}$. The following result can be easily verified using the canonical trivalization $T\cM\cong\frakg[1]\times\frakg[1]$.

\begin{lemma}
Let $(\cM, Q)=(\frakg[1], d_{\CE})$ be the canonical dg-manifold corresponding to a finite-dimensional Lie algebra $\frakg$. 
Then,
\[ H^k( \Gamma (T^\vee \cM \otimes \End  T\cM), Q)\cong 
H^{k-1}_{\CE} (\frakg, \frakg^\vee\otimes \frakg^\vee \otimes \frakg) ,\] 
and 
\[ H^k(\Omega^k(\cM),Q)\cong(S^k\frakg^\vee)^\frakg .\]
\end{lemma}

\begin{e-proposition}
Let $(\cM, Q)=(\frakg[1], d_{\CE})$ be the canonical dg-manifold corresponding to a finite-dimensional Lie algebra $\frakg$. 
Then the Atiyah class $\Aat_{\frakg[1]}$ is precisely the Lie bracket of $\frakg$ regarded as an element of 
$(\frakg^\vee\otimes\frakg^\vee\otimes\frakg)^\frakg
\isomorphism H^1\big(\sections{T^\vee\cM\otimes\End T\cM},Q\big)$.
Consequently, the isomorphism 
\[ \prod_k H^k\big(\Omega^k(\cM),Q\big)\xto{\isomorphism}
\big(\widehat{S}(\frakg^\vee)\big)^{\frakg} \] 
maps the Todd class $\Todd_{\frakg[1]}$ onto the Duflo element of $\frakg$.
\end{e-proposition}

\begin{example}
Consider a dg-manifold of the form $\cM=(\RR^{m|n},Q)$.
Let $(x_1,\cdots,x_m;x_{m+1}\cdots x_{m+n})$
be coordinate functions on $\RR^{m|n}$, and write
$Q=\sum_k Q_k(x)\frac{\partial}{\partial x_k}$.
Then the Atiyah $1$-cocycle associated to the trivial connection 
$\nabla_{\frac{\partial}{\partial x_i}}\frac{\partial}{\partial x_j}=0$ 
is given by
\begin{equation}\label{velociraptor}
 \At_{\cM}\left(\frac{\partial}{\partial x_i},\frac{\partial}{\partial x_j}\right)
=(-1)^{|x_i|+|x_j|}\sum_k\frac{\partial^2 Q_k}{\partial x_i\partial x_j}
\frac{\partial}{\partial x_k} 
\end{equation}
As we can see from \eqref{velociraptor}, the Atiyah 1-cocycle $\At_\cM$ includes the information about the homological vector field of second-order and higher. 
\end{example}

\section{Atiyah class and homotopy Lie algebras}

Let $\cM$ be a graded manifold. A $(1,2)$-tensor of degree $k$ on $\cM$ is a $\CC$-linear map $\alpha: \vect(\cM) \otimes_{\mathbb{C}} \vect(\cM) \to \vect(\cM)$
such that $|\alpha(X,Y)| = |X| + |Y| + k$ and
\[ \alpha(fX,Y) = (-1)^{k|f|} f \alpha(X,Y) = (-1)^{|f||X|}\alpha(X,fY).\]
We denote the space of  $(1,2)$-tensors of degree $k$ by $\T^{1,2}_k(\cM)$, 
and the space of all $(1,2)$-tensors by $\T^{1,2}(\cM) = \bigoplus_k \T^{1,2}_k(\cM)$.

The torsion of an affine  connection $\nabla$ is given by
\begin{equation}
T(X,Y) = \nabla_X Y - (-1)^{|X||Y|} \nabla_Y X - [X,Y].
\end{equation}
The torsion is an element in $\T^{1,2}_0(\cM)$.
Given any affine connection, one can define its opposite
affine connection $\nabla^\op$, given by
\begin{equation}
\nabla^\op_X Y = \nabla_X Y - T(X,Y) = [X,Y] + (-1)^{|X||Y|} \nabla_Y X.
\end{equation}
The average $\frac{1}{2}(\nabla+\nabla^\op)$ is a torsion-free
affine connection. This shows that torsion-free affine
connections always exist on graded manifolds. 

In this section, we show that, as in the classical situation considered by Kapranov in~\cite{Kapranov,LSX:14}, the Atiyah $1$-cocycle of a dg-manifold gives rise to  
an interesting homotopy Lie algebra.
As in the last section, let $(\cM,Q)$ be a dg-manifold and let $\nabla$ be an affine connection on $\cM$. 
The following can be easily verified by direct computation.
\begin{enumerate}
\item The anti-symmetrization of the Atiyah $1$-cocycle $\At_\cM$ is equal to 
$L_Q T$, so $\At_\cM$ is graded antisymmetric up to an exact term. 
In particular, if $\nabla$ is torsion-free, we have 
\[ \At_\cM(X,Y)=(-1)^{\degree{X}\degree{Y}}\At_\cM(Y,X) .\] 
\item The degree $1+\degree{X}$ operator $\At_\cM(X,-)$ 
need not be a derivation of the degree $+1$ `product' $\XX(\cM)\otimes_{\mathbb{C}}\XX(\cM)\xto{\At_\cM}\XX(\cM)$. 
However, the Jacobiator
\begin{multline*} (X,Y,Z)\mapsto 
\At_\cM\big(X,\At_\cM(Y,Z)\big) - 
\big\{
(-1)^{\degree{X}+1}\At_\cM\big(\At_\cM(X,Y),Z\big)
\\ +(-1)^{(\degree{X}+1)(\degree{Y}+1)}\At_\cM\big(Y,\At_\cM(X,Z)\big)
\big\} ,
\end{multline*}
of $\At_\cM$, which vanishes precisely when $\At_\cM(X,-)$ is a derivation of $\At_\cM$, 
is equal to $\pm L_Q(\nabla\At_\cM)$. Hence $\At_\cM$ satisfies the graded Jacobi identity up to the exact term $L_Q(\nabla\At_\cM)$.
\end{enumerate}

Armed with this motivation, we can now state the main result of this note.
\begin{theorem}\label{horse}
Let $(\cM,Q)$ be a dg-manifold and let $\nabla$ be a torsion-free affine connection on $\cM$.
There exists a sequence $(\lambda_k)_{k\geq 2}$ of maps $\lambda_k\in\Hom(S^k (T\cM),T\cM[-1])$ starting with 
$\lambda_2:=\At_\cM\in\Hom(S^2(T\cM),T\cM[-1])$ 
which, together with $\lambda_1:=L_Q:\XX(\cM)\to\XX(\cM)$, 
satisfy the $L_\infty[1]$-algebra axioms.
As a consequence, the space of vector fields $\XX(\cM)$
on a dg-manifold $(\cM,Q)$ admits an $L_\infty[1]$-algebra structure 
with the Lie derivative $L_Q$ as unary bracket $\lambda_1$ 
and the Atiyah cocycle $\At_\cM$ as binary bracket $\lambda_2$.
\end{theorem}

To prove Theorem~\ref{horse}, we introduce a Poincar\'e--Birkhoff--Witt map for graded manifolds.

It was shown in~\cite{LSX:14} that every torsion-free affine connection 
$\nabla$ on a smooth manifold $M$ determines an isomorphism of coalgebras (over $C^\infty(M)$) 
\begin{equation}
\PBW:\sections{S(TM)}\xto{\isomorphism}D(M),
\end{equation}
called the Poincar\'e--Birkhoff--Witt (PBW) map. Here $D(M)$ denotes the space of differential operators on $M$.

Geometrically, an affine connection $\nabla$ induces an exponential map
$TM\to M\times M$, which is a well-defined diffeomorphism
from a neighborhood of the zero section of $TM$ to a neighborhood of 
the diagonal $\Delta(M)$ of $M\times M$.
Sections of $S(TM)$ can be viewed as fiberwise distributions on $TM$ supported on the zero section, while $D(M)$
can be viewed as the space of source-fiberwise distributions on $M\times M$ supported on the diagonal $\Delta (M)$. The map 
$\PBW:\sections{S(TM)}\to D(M)$
is simply the push-forward of fiberwise distributions through 
the exponential map $\exp^\nabla:TM\to M\times M$ and is clearly an isomorphism of coalgebras over $C^\infty (M)$. 

Even though, for a \emph{graded} manifold $\cM$ endowed with a torsion-free affine connection $\nabla$, we lack an exponential map 
$\exp^\nabla:\T\cM\to\cM\times\cM$, a PBW map can still be defined
purely algebraically thanks to the iteration formula introduced
in~\cite{LSX:14}.

\begin{lemma}
Let $\cM$ be a $\ZZ$-graded manifold and let $\nabla$ be a torsion-free 
affine connection on $\cM$. The Poincar\'e-Birkhoff-Witt map
inductively defined by the relations\footnote{We would like to thank Hsuan-Yi Liao for correcting 
a sign error in the inductive formula defining the map $\PBW$.}
\begin{gather*} 
\PBW(f)=f, \quad\forall f\in C^\infty(\cM); \\
\PBW(X)=X, \quad\forall X\in\XX(\cM); 
\end{gather*}
and
\begin{multline*}
\PBW(X_0\odot \cdots\odot X_n)=\frac{1}{n+1} \sum_{k=0}^n 
(-1)^{\degree{X_k}(\degree{X_0}+\cdots+\degree{X_{k-1}})}
\big\{ X_k\cdot\PBW(X_0\odot\cdots\odot \widehat{X}_k\odot\cdots\odot X_n) 
\\ -\PBW\big(\nabla_{X_k}(X_0\odot\cdots\odot \widehat{X}_k\odot\cdots\odot X_n)\big)\big\}, 
\end{multline*}
for all $n\in\NN$ and $X_0,\dots,X_n\in\XX(\cM)$, establishes an isomorphism 
\begin{equation}\label{eq:PBW}
\PBW:\sections{S(T\cM)}\xto{\isomorphism}D(\cM) 
.\end{equation}
of coalgebras over $C^\infty(\cM)$. \end{lemma}

Now assume that $(\cM, Q)$ is a dg-manifold.
The homological vector field $Q$ induces a degree $+1$ coderivation of $D(\cM)$ defined by the Lie derivative:
\begin{equation}
L_Q(X_1\cdots X_n)=\sum_{k=1}^n (-1)^{|X_1|+\cdots +|X_{k-1}|}
X_1\cdots X_{k-1}[Q,X_k]X_{k+1}\cdots X_n 
.\end{equation}

Now using the isomorphism of coalgebras $\PBW$ as in Eq.~\eqref{eq:PBW}
to transfer $L_Q$ from $D(\cM)$ to $\sections{S(T\cM)}$, we obtain $\delta:=(\PBW)^{-1}\circ L_Q\circ\PBW$, a 
degree $1$ coderivation of $\sections{S(T\cM)}$.
Finally, dualizing $\delta$, we obtain an operator
\[ D:\sections{\hatS(T^\vee\cM)}\to\sections{\hatS(T^\vee\cM)} \]
as \[ \sections{\hatS(T^\vee\cM)}\cong
\Hom_{C^\infty(\cM)}(\sections{S(T\cM)},C^\infty(\cM)) .\]
\begin{theorem}\label{zebra}
Let $(\cM,Q)$ be a dg-manifold and let $\nabla$ be a torsion-free affine connection on $\cM$.
\begin{enumerate}
\item The operator $D$, dual to $(\PBW)^{-1}\circ L_Q\circ\PBW$, is a degree $+1$ derivation of the graded algebra $\sections{\widehat{S}(T^\vee\cM)}$ satisfying $D^2=0$. 
\item There exists a sequence $\{R_k\}_{k\geq 2}$ of homomorphisms $R_k\in\Hom(S^k T\cM,T\cM[-1])$, 
whose first term $R_2$ is precisely the Atiyah $1$-cocycle $\At_\cM$, 
such that $D=L_Q+\sum_{k=2}^\infty\widetilde{R_k}$, where $\widetilde{R_k}$ denotes 
the $C^\infty (\cM)$-linear operator on $\sections{\widehat{S}(T^\vee\cM)}$ corresponding to $R_k$. 
\end{enumerate}
\end{theorem}

Finally we note that Theorem~\ref{horse} 
is a consequence of Theorem~\ref{zebra}.

\section*{Acknowledgements}
We would like to thank several institutions for their hospitality 
while work on this project was being done:
Penn State University (Mehta),  and Universit\'e Paris Diderot (Xu).
We also wish to thank Hsuan-Yi Liao, Dmitry Roytenberg and Boris Shoikhet for inspiring discussions.

\bibliographystyle{amsplain}
\bibliography{mathscinet}
\end{document}